\documentclass[a4paper]{article}
\pdfoutput=1
\usepackage{hyperref}
\hypersetup{hypertexnames = false, bookmarksdepth = 2, bookmarksopen = true, colorlinks, linkcolor = black, citecolor = black, urlcolor = black, pdfstartview={XYZ null null 1}}

\usepackage{amsfonts}
\usepackage[fleqn, leqno]{amsmath}
\usepackage{amsthm}

\usepackage[capitalise]{cleveref}

\usepackage[backend=biber, maxbibnames=99, datamodel=mrnumber, sortcites]{biblatex}

\usepackage{booktabs}
\usepackage{colonequals}
\usepackage{diagbox}
\usepackage{enumitem}
\usepackage{fullpage}
\usepackage{mathtools}
\usepackage{parskip}
\usepackage{thmtools}
\usepackage{tikz-cd}
\usepackage[colorinlistoftodos, textsize = footnotesize]{todonotes}
\usepackage{xparse}
\usepackage{xspace}

\usepackage[frozencache]{minted}

\usepackage[utf8]{inputenc}
\usepackage[T1]{fontenc}
\usepackage{libertine}
\usepackage[libertine]{newtxmath}
\usepackage[scaled=0.8]{FiraMono}
\usepackage{eucal}
\usepackage{microtype}
\usepackage{gitinfo2}

\DeclareFieldFormat{mrnumber}{%
  MR\addcolon\space
  \ifhyperref
    {\href{http://www.ams.org/mathscinet-getitem?mr=#1}{\nolinkurl{#1}}}
    {\nolinkurl{#1}}}

\renewbibmacro*{doi+eprint+url}{%
  \iftoggle{bbx:doi}
    {\printfield{doi}}
    {}%
  \newunit\newblock
  \printfield{mrnumber}%
  \newunit\newblock
  \iftoggle{bbx:eprint}
    {\usebibmacro{eprint}}
    {}%
  \newunit\newblock
  \iftoggle{bbx:url}
    {\usebibmacro{url+urldate}}
    {}}

\newcounter{todocounter}
\DeclareDocumentCommand\addreference{g}{\stepcounter{todocounter}\todo[color = blue!30]{\thetodocounter. Add reference\IfNoValueF{#1}{: #1}}\xspace}
\DeclareDocumentCommand\checkthis{g}{\stepcounter{todocounter}\todo[color = red!50]{\thetodocounter. Check this\IfNoValueF{#1}{: #1}}\xspace}
\DeclareDocumentCommand\fixthis{g}{\stepcounter{todocounter}\todo[color = orange!50]{\thetodocounter. Fix this\IfNoValueF{#1}{: #1}}\xspace}
\DeclareDocumentCommand\expand{g}{\stepcounter{todocounter}\todo[color = green!50]{\thetodocounter. Expand\IfNoValueF{#1}{: #1}}\xspace}

\declaretheoremstyle[
  spaceabove = 3pt,
  spacebelow = 3pt,
]{first}
\declaretheoremstyle[
  spaceabove = 3pt,
  spacebelow = 3pt,
]{second}
\declaretheorem[numberwithin=section, style=first]{theorem}

\declaretheorem[sibling=theorem, style=first]{corollary}
\declaretheorem[sibling=theorem, style=first]{lemma}
\declaretheorem[sibling=theorem, style=first]{proposition}

\declaretheorem[sibling=theorem, style=second]{example}
\declaretheorem[sibling=theorem, style=second]{remark}
\declaretheorem[sibling=theorem, style=second]{definition}

\declaretheorem[sibling=theorem, style=second]{construction}

\crefname{construction}{Construction}{Constructions}

\declaretheorem[numberwithin=section, style=first, title=Theorem]{alphatheorem}
\declaretheorem[sibling=alphatheorem, style=first, title=Assumption]{alphaassumption}

\crefname{alphatheorem}{Theorem}{Theorems}
\crefname{alphaconjecture}{Conjecture}{Conjectures}
\crefname{alphacorollary}{Corollary}{Corollaries}
\crefname{alphaproposition}{Proposition}{Propositions}
\crefname{alphaassumption}{Assumption}{Assumptions}

\makeatletter
\def\gitfootnote{\gdef\@thefnmark{}\@footnotetext}
\makeatother

\mathchardef\mhyphen="2D
\newcommand\dash{\nobreakdash-\hspace{0pt}}

\DeclareMathOperator\At{At}
\DeclareMathOperator\at{at}
\DeclareMathOperator\chern{c} 
\DeclareMathOperator\Chern{ch} 
\DeclareMathOperator\chow{CH} 
\DeclareMathOperator\degeneracy{D}
\DeclareMathOperator\Ext{Ext}
\DeclareMathOperator\GL{GL}
\DeclareMathOperator\Gr{Gr}
\DeclareMathOperator\HH{H}
\DeclareMathOperator\HS{HS}
\DeclareMathOperator\hh{h}
\DeclareMathOperator\Hom{Hom}

\DeclareMathOperator\coker{coker}

\DeclareMathOperator\Pic{Pic}
\DeclareMathOperator\Spec{Spec}
\DeclareMathOperator\rk{rk}
\DeclareMathOperator\id{id}
\DeclareMathOperator\source{s}

\DeclareMathOperator\target{t}
\DeclareMathOperator\todd{td}
\DeclareMathOperator\dimvect{\underline{\mathrm{dim}}}

\newcommand\can{\ensuremath{\mathrm{can}}}
\newcommand\Gm{\ensuremath{\mathbb{G}_{\mathrm{m}}}}
\newcommand\kodairaspencer{\ensuremath{\mathrm{KS}}}
\newcommand\linear{\ensuremath{\mathrm{lin}}}
\newcommand\pt{\ensuremath{\mathrm{pt}}}
\newcommand\tangent{\ensuremath{\mathrm{T}}}
\newcommand\tautological{\ensuremath{\mathrm{taut}}}
\newcommand\weyl{\ensuremath{\mathrm{W}}}
\newcommand\sheafHom{\ensuremath{\mathcal{H}\mathrm{om}}}
\newcommand\sheafEnd{\ensuremath{\mathcal{E}\mathrm{nd}}}
\newcommand\sheafExt{\ensuremath{\mathcal{E}\mathrm{xt}}}

\DeclareDocumentCommand\modulistack{om}{\IfNoValueTF{#1}{\mathcal{M}{(#2)}}{\mathcal{M}^{#1}(#2)}}
\DeclareDocumentCommand\modulispace{om}{\IfNoValueTF{#1}{\mathrm{M}{(#2)}}{\mathrm{M}^{#1}(#2)}}
\newcommand\kronecker[2]{\smash{\mathrm{M}_{#2}^{#1}}}

\DeclareDocumentCommand\representationvariety{om}{\IfNoValueTF{#1}{\mathrm{R}({#2})}{\mathrm{R}^{#1}{(#2)}}}
\newcommand\group[1]{\mathrm{G}_{#1}}
\newcommand\PG[1]{\mathrm{PG}_{#1}}

\newcommand\semistable[1]{#1\mhyphen\mathrm{sst}}
\newcommand\stable[1]{#1\mhyphen\mathrm{st}}

\title{On Chow rings of quiver moduli}

\author{Pieter Belmans \and Hans Franzen}

\begin{document}
\maketitle


\begin{abstract}
  We describe the point class and Todd class
  in the Chow ring of a quiver moduli space,
  building on a result of Ellingsrud--Str\o mme.
  This, together with the presentation of the Chow ring by the second author,
  makes it possible to compute integrals on quiver moduli.
  To do so
  we construct a canonical morphism of universal representations in great generality,
  and along the way point out its relation to the Kodaira--Spencer morphism.

  We illustrate the results
  by computing some invariants of some ``small'' Kronecker moduli spaces.
  We also prove that the first non-trivial (6-dimensional) Kronecker quiver moduli space
  is isomorphic to
  the zero locus of a general section of~$\mathcal{Q}^\vee(1)$ on~$\Gr(2,8)$.
\end{abstract}

\tableofcontents

\section{Introduction}

In \cite{MR3318266} the second author gave a tautological presentation of
the Chow ring (with rational coefficients) of a quiver moduli space.
This completed the work of King--Walter \cite[Theorem~3]{MR1324213}
which showed that it was generated by Chern classes of tautological bundles,
adapting a result of Ellingsrud--Str\o mme \cite{MR1228610}
(which by their own attribution is folklore).
Two missing ingredients, necessary for computations,
are
\begin{itemize}
  \item an expression for the point class,
  \item an expression for the Todd class.
\end{itemize}
In this article we provide precisely such an expression,
further elaborating the method of Ellingsrud--Str\o mme
and making it explicit for quiver moduli.

\paragraph{Setup}
Let~$k$ be an algebraically closed field.
The following assumptions will be used repeatedly.
\begin{alphaassumption}
  \label{assumption:standing}
  Let~$Q=(Q_0,Q_1)$ be an acyclic quiver,
  $\mathbf{d}\in\mathbb{N}^{Q_0}$ a dimension vector,
  and~$\theta\in\Hom(\mathbb{Z}^{Q_0},\mathbb{Z})$ a stability parameter
  such that~$\mathbf{d}$ is $\theta$-coprime.
\end{alphaassumption}
Let~$X=\modulispace[\theta]{Q,\mathbf{d}}$ denote
the moduli space of~$\theta$-(semi)stable representations of~$Q$ with dimension vector~$\mathbf{d}$,
and~$\mathcal{U}=\bigoplus_{i\in Q_0}\mathcal{U}_i$ on~$X$ the universal representation.

The first theorem expresses the point class in terms of the Chern classes of the universal representation.
\begin{alphatheorem}
  \label{theorem:point-class}
  With the setup and notation from above the following hold.
  \begin{enumerate}
    \item
      \label{enumerate:point-class-1}
      We can write a closed point~$x_0\in X$ as the degeneracy locus
      \begin{equation}
        \{x_0\}=\degeneracy_{e-1}(j_1^*\sigma)=\degeneracy_{e-1}(j_2^*\sigma)
      \end{equation}
      where
      \begin{equation}
        \sigma\colon\bigoplus_{i\in Q_0}\mathcal{U}_i^\vee\boxtimes\mathcal{U}_i\to\bigoplus_{a\in Q_1}\mathcal{U}_{\source(a)}^\vee\boxtimes\mathcal{U}_{\target(a)}
      \end{equation}
      is the tautological morphism defined in \cref{subsection:tautological-morphism},
      $j_1$ resp.~$j_2$ are the inclusions~$X\hookrightarrow X\times X$ given by~$x\mapsto(x,x_0)$ resp.~$x\mapsto(x_0,x)$
      and~$e=\sum_{i\in Q_0}d_i^2=\rk\bigoplus_{i\in Q_0}\mathcal{U}_i^\vee\otimes\mathcal{U}_i$.
    \item
      \label{enumerate:point-class-2}
      Consequently, in~$\chow^{\dim X}(X)$
      the point class~$[\pt_X]$ is equal to
      the projection of the classes
      \begin{equation}
        \label{equation:identity}
        \prod_{a\in Q_1}\chern\left( \mathcal{U}_{\source(a)}^\vee \right)^{d_{\target(a)}}
        /
        \prod_{i\in Q_0}\chern\left( \mathcal{U}_i^\vee \right)^{d_i} \\
        =
        \prod_{a\in Q_1}\chern\left( \mathcal{U}_{\target(a)} \right)^{d_{\source(a)}}
        /
        \prod_{i\in Q_0}\chern\left( \mathcal{U}_i \right)^{d_i}
      \end{equation}
      onto the homogeneous component of degree~$\dim X$.
  \end{enumerate}
\end{alphatheorem}

Its proof,
given in \cref{subsection:points},
is obtained by adapting and building upon \cite[Lemma~2.4]{MR1228610}
(which is for stable sheaves on~$\mathbb{P}^2$)
to an explicit expression in the case of quiver moduli.

The next theorem gives an expression of the Todd class,
necessary for Hirzebruch--Riemann--Roch computations on quiver moduli.

\begin{alphatheorem}
  \label{theorem:todd-class}
  With the setup and notation from above we have
  \begin{equation}
    \label{equation:todd-class}
    \todd_X=\prod_{a\in Q_1}\todd\left( \mathcal{U}_{\source(a)}^\vee\otimes\mathcal{U}_{\target(a)} \right)\prod_{i\in Q_0}\todd\left( \mathcal{U}_i^\vee\otimes\mathcal{U}_i \right)^{-1}.
  \end{equation}
\end{alphatheorem}

Its proof,
given in \cref{subsection:todd-class},
is obtained by using the 4-term tangent bundle sequence \eqref{equation:4-term-tangent-sequence} on~$\modulispace[\theta]{Q,\mathbf{d}}$.
This sequence is itself not new (see, e.g.,
\cite[Remark~4.1]{MR4352662},
or the discussion of a special case in \cref{remark:ellingsrud-stromme})
but we give a uniform and general treatment of its construction.

\paragraph{On the tautological and Kodaira--Spencer morphism}
Both proofs use the tautological morphism,
which can be constructed in complete generality on the product of two different moduli stacks (resp.~fine moduli spaces),
as explained in \cref{subsection:tautological-morphism}.
By considering different restrictions
(to the fiber of a projection, resp.~the diagonal)
we will find the necessary ingredients for the proofs.

An important result of independent interest related to the tautological morphism is
the construction of a Kodaira--Spencer type morphism \eqref{equation:kodaira-spencer}
and the proof that it is an isomorphism, see \cref{proposition:kodaira-spencer}.

\paragraph{On Kronecker quiver moduli}
To illustrate the methods
we will (re)derive some results for Kronecker moduli.
These are quiver moduli associated to the generalised Kronecker quiver,
and their study was initiated by Drezet in \cite{MR0944602,MR0916199,MR0915184},
in relation to moduli spaces of sheaves on~$\mathbb{P}^2$.
They are often isomorphic to a Grassmannian,
and there exist various duality and periodicity isomorphism amongst them,
see \cref{subsection:census}.
In \cref{proposition:kronecker-invariants} (and more precisely \cref{table:kronecker-census})
we use the results from \cref{theorem:point-class,theorem:todd-class}
to describe invariants of the four smallest Kronecker moduli
which are not Grassmannians.

For the smallest case,
i.e.,
the 6-dimensional moduli space~$\modulispace[\theta_\can]{Q}{(2,3)}$
where~$Q$ is the~3\dash Kronecker quiver,
the numerical invariants agree with
those of the 6-dimensional version of the variety~(b11)
studied by K\"uchle in \cite{MR1326986}.
In \cref{proposition:identification}
we show that they are indeed isomorphic\footnote{
In the first preprint version of this article
this was phrased as a conjecture,
which we can now prove.
}.

\paragraph{Acknowledgements}
We would like to thank Enrico Fatighenti and Fabio Tanturri
for finding the variety in \cref{construction:zero-locus}
when prompted with the question whether there is a~6\dash dimensional zero locus of an equivariant vector bundle on a Grassmannian
with the given Betti numbers.
We want to thank Victoria Hoskins for comments on a preliminary version of the paper.
After the first version of this paper was posted on the arXiv,
discussions with Laurent Manivel, Fabian Reede, and Jieao Song
led to the proof of the identification in \cref{proposition:identification}.

The first author was partially supported by the Luxembourg National Research Fund (FNR--17113194).
The second author was partially supported by the Deutsche Forschungsgemeinschaft (DFG, German Research Foundation) -- SFB-TRR 358/1 2023 -- 491392403.

\section{Quiver moduli and their Chow ring}
We first establish the notation and some basic properties for quivers and their moduli of representations.
For more information the reader is referred to \cite{MR2484736},
and \cite{2210.00033v1} for the stacky point-of-view.

Let $Q=(Q_0,Q_1)$ be a quiver.
For an arrow~$a \in Q_1$,
we denote with $\source(a)$ its source and with $\target(a)$ its target.
We denote the path algebra by $kQ$ and we freely identify representations of $Q$ over $k$ with left $kQ$-modules.
We will denote the Euler form of $Q$ by~$\langle-,-\rangle$.

For a fixed dimension vector $\mathbf{d}\in\mathbb{N}^{Q_0}$,
we fix $k$-vector spaces~$U_i$ of dimension~$d_i$
and consider the~$k$-vector space
\begin{equation}
  \representationvariety{Q,\mathbf{d}} = \bigoplus_{a \in Q_1} \Hom(U_{\source(a)},U_{\target(a)}).
\end{equation}
We regard it as an affine space.
Its $k$-valued points are representations of~$Q$ over~$k$ on the vector spaces~$U_i$.
Moreover, we define
\begin{equation}
  \group{\mathbf{d}} = \prod_{i \in Q_0} \GL(U_i).
\end{equation}
We define a left action of~$\group{\mathbf{d}}$ on~$\representationvariety{Q,\mathbf{d}}$
by letting $g = (g_i)_{i \in Q_0}$ act on $M = (M_a)_{a \in Q_1}$ by
\begin{equation}
  g \cdot M = (g_{\target(a)}M_ag_{\source(a)}^{-1}).
\end{equation}
Two points of $\representationvariety{Q,\mathbf{d}}$ are isomorphic as representations
if and only if they lie in the same~$\group{\mathbf{d}}$-orbit.
Note that the image of the embedding~$\Gm\hookrightarrow \group{\mathbf{d}}$,
which on~$k$-valued points is given by $t \mapsto t \cdot \id$,
acts trivially on $\representationvariety{Q,\mathbf{d}}$.
Therefore the action descends to an action of the cokernel~$\PG{\mathbf{d}}$ of~$\Gm\hookrightarrow\group{\mathbf{d}}$.

\begin{definition}
  The quotient stack $\modulistack{Q,\mathbf{d}} = [\representationvariety{Q,\mathbf{d}}/\group{\mathbf{d}}]$
  is called the \emph{moduli stack} of representations of dimension vector $\mathbf{d}$.
\end{definition}

We introduce a notion of stability.
A \emph{stability parameter} is some $\theta \in \Hom(\mathbb{Z}^{Q_0},\mathbb{Z})$
such that $\theta(\mathbf{d})= 0$.
We will use the basis~$e_i^*$ for~$\Hom(\mathbb{Z}^{Q_0},\mathbb{Z})$.

\begin{definition}
  A representation $M$ of dimension vector $\mathbf{d}$
  is called $\theta$-\emph{semistable} if $\theta(\dimvect M') \leq 0$
  for all proper non-zero subrepresentations $M'$ of $M$.
  We say $M$ is $\theta$-\emph{stable} if the inequality is strict for every such $M'$.
\end{definition}

Using the Hilbert--Mumford criterion,
King showed in \cite{MR1315461} that this notion of (semi\nobreakdash-)stability
agrees with (semi\nobreakdash-)stability with respect to the $\group{\mathbf{d}}$-linearized line bundle $L(\theta)$,
which, as a line bundle, is trivial,
and whose $\group{\mathbf{d}}$-action on the fiber
is given by the character~$\chi_\theta(g) = \prod_{i \in Q_0} \det(g_i)^{-\theta_i}$.

Let
\begin{equation}
  \label{equation:double-inclusion}
  \representationvariety[\semistable{\theta}]{Q,\mathbf{d}}
  \subseteq \representationvariety[\stable{\theta}]{Q,\mathbf{d}}
  \subseteq \representationvariety{Q,\mathbf{d}}
\end{equation}
be the loci of $\theta$-semistable and $\theta$-stable representations respectively.
They are Zariski open, but possibly empty.

\begin{definition}
  The $\theta$-semistable and $\theta$-stable \emph{moduli stacks} of representations of $Q$ of dimension vector $\mathbf{d}$ are defined as
  \begin{equation}
    \begin{aligned}
      \modulistack[\semistable{\theta}]{Q,\mathbf{d}}
      &\colonequals[\representationvariety[\semistable{\theta}]{Q,\mathbf{d}}/\group{\mathbf{d}}] \\
      \modulistack[\stable{\theta}]{Q,\mathbf{d}}
      &\colonequals[\representationvariety[\stable{\theta}]{Q,\mathbf{d}}/\group{\mathbf{d}}].
    \end{aligned}
  \end{equation}
  We define the $\theta$-semistable and $\theta$-stable \emph{moduli spaces} as
  \begin{equation}
    \begin{aligned}
      \modulispace[\semistable{\theta}]{Q,\mathbf{d}}
      &\colonequals\representationvariety{Q,\mathbf{d}}/\!\!/_{L(\theta)}\PG{\mathbf{d}} \\
      \modulispace[\stable{\theta}]{Q,\mathbf{d}}
      &\colonequals\representationvariety{Q,\mathbf{d}}/_{L(\theta)}\PG{\mathbf{d}}.
    \end{aligned}
  \end{equation}
\end{definition}

\begin{proposition}
  \label{proposition:basic-properties}
  We have a diagram
  \begin{equation}
    \label{equation:big-diagram}
    \begin{tikzcd}
      \representationvariety[\stable{\theta}]{Q,\mathbf{d}} \arrow[d, two heads] \arrow[r, hook]
      & \representationvariety[\semistable{\theta}]{Q,\mathbf{d}} \arrow[d, two heads] \arrow[r, hook]
      & \representationvariety{Q,\mathbf{d}} \arrow[d, two heads] \\
      \modulistack[\stable{\theta}]{Q,\mathbf{d}} \arrow[r, hook] \arrow[d]
      & \modulistack[\semistable{\theta}]{Q,\mathbf{d}} \arrow[r, hook] \arrow[d]
      & \modulistack{Q,\mathbf{d}} \arrow[d] \\
      \modulispace[\stable{\theta}]{Q,\mathbf{d}} \arrow[r, hook]
      & \modulispace[\semistable{\theta}]{Q,\mathbf{d}} \arrow[r]
      & \modulispace{Q,\mathbf{d}}
    \end{tikzcd}
  \end{equation}
  where all but one of the horizontal morphisms are open immersions,
  the top vertical morphisms are quotient stack presentations,
  and the bottom vertical morphisms are good (resp.~adequate in positive characteristic) moduli spaces
  in the sense of Alper \cite{MR3237451,MR3272912}.

  Moreover,
  \begin{enumerate}
    \item\label{enumerate:basic-dimension}
      If~$\modulispace[\stable{\theta}]{Q,\mathbf{d}}$ is non-empty,
      then the dimension of~$\modulispace[\semistable{\theta}]{Q,\mathbf{d}}$
      and~$\modulispace[\stable{\theta}]{Q,\mathbf{d}}$ is~$1-\langle\mathbf{d},\mathbf{d}\rangle$.
    \item\label{enumerate:basic-projective}
      The morphism~$\modulispace[\semistable{\theta}]{Q,\mathbf{d}} \to \modulispace{Q,\mathbf{d}}$
      is projective.
      Thus, if~$Q$ is acyclic,
      then~$\modulispace[\semistable{\theta}]{Q,\mathbf{d}}$ is projective.
    \item\label{enumerate:basic-smooth}
      The variety~$\modulispace[\stable{\theta}]{Q,\mathbf{d}}$ is smooth.
  \end{enumerate}
\end{proposition}

We say that the dimension vector~$\mathbf{d}$ is \emph{$\theta$-coprime}
if~$\theta(\mathbf{d}') \neq 0$ for all~$0 \lneqq \mathbf{d}' \lneqq \mathbf{d}$.
Note that if $\mathbf{d}$ is $\theta$-coprime,
then $\mathbf{d}$ is necessarily indivisible.
If~$\mathbf{d}$ is assumed to be indivisible
then for~$\theta$ not on a wall
holds that~$\mathbf{d}$ is~$\theta$-coprime, see \cite[\S3.5]{MR2484736}.

The following corollary, together with \cref{corollary:descent}, explains our standing assumptions.
\begin{corollary}
  \label{corollary:smooth-projective-quiver-moduli}
  Assume \cref{assumption:standing}.
  Then every $\theta$-semistable representation is $\theta$-stable,
  and thus~$\modulispace[\stable{\theta}]{Q,\mathbf{d}}$ is a smooth projective variety.
\end{corollary}

\subsection{Universal representations on moduli stacks and moduli spaces}
\label{subsection:universal-bundles}
On the representation variety~$\representationvariety{Q,\mathbf{d}}$
we consider the trivial vector bundle with fiber $U_i$
and equip it with the action of $\group{\mathbf{d}}$
by letting its $i$th factor $\GL(U_i)$ act by matrix-vector multiplication.
Denote this $\group{\mathbf{d}}$-equivariant vector bundle also $U_i$.
By abuse of notation, we denote its restriction to the $\theta$-(semi\nobreakdash-)stable locus
in the top row of \eqref{equation:big-diagram}
with the same symbol.
The induced vector bundles on the respective quotient stacks in the middle row of \eqref{equation:big-diagram}
will be denoted~$\mathcal{U}_i$.

For every~$a\in Q_1$
we consider the homomorphism~$U_a\colon U_{\source(a)} \to U_{\target(a)}$ of sheaves,
which, on the fiber over a point corresponding to the representation~$V$
is defined by sending~$v \in V_{\source(a)}$ to $V_a(v) \in V_{\target(a)}$.
This homomorphism is~$\group{\mathbf{d}}$-equivariant
and thus descends to a homomorphism $\mathcal{U}_a\colon\mathcal{U}_{\source(a)}\to\mathcal{U}_{\target(a)}$.
\begin{definition}
  The collection~$\mathcal{U} = ((\mathcal{U}_i)_{i \in Q_0},(\mathcal{U}_a)_{a \in Q_1})$
  on any of the stacks in the middle row of \eqref{equation:big-diagram}
  is called the \emph{universal representation}.
\end{definition}

Now assume that $\mathbf{d}$ is indivisible.
Then it is possible to construct a universal representation also on the stable moduli space,
as was done by King \cite[Proposition~5.3]{MR1315461}.
To this end, fix some $\mathbf{a} \in \mathbb{Z}^{Q_0}$ such that $\theta(\mathbf{a})=1$.
On the $\group{\mathbf{d}}$-equivariant vector bundle $U_i(\mathbf{a}) = U_i \otimes L(\mathbf{a})$
living on~$\representationvariety[\stable{\theta}]{Q,\mathbf{d}}$
the kernel of $\group{\mathbf{d}} \to \PG{\mathbf{d}}$ acts trivially.
It is therefore a $\PG{\mathbf{d}}$-equivariant vector bundle
and thus descends to the geometric $\PG{\mathbf{d}}$-quotient $\modulispace[\stable{\theta}]{Q,\mathbf{d}}$.

Call the descended bundle $\mathcal{U}_i(\mathbf{a})$.
We moreover obtain homomorphisms $\mathcal{U}(\mathbf{a})_a\colon\mathcal{U}(\mathbf{a})_{\source(a)} \to \mathcal{U}(\mathbf{a})_{\target(a)}$
in the same way as for the universal representation on the moduli stack.



\begin{corollary}
  \label{corollary:descent}
  Assume \cref{assumption:standing}.
  Let~$\mathbf{a}$ be an element of~$\mathbb{Z}^{Q_0}$
  such that~$\theta(\mathbf{a})=0$.
  On $\modulispace[\stable{\theta}]{Q,\mathbf{d}}$ there exists
  the universal representation~$\mathcal{U}(\mathbf{a}) = ((\mathcal{U}(\mathbf{a})_i)_{i \in Q_0},(\mathcal{U}(\mathbf{a})_a)_{a \in Q_1})$.
\end{corollary}

\begin{remark}
  Without the indivisibility condition
  it is shown in \cite[Theorem~4.4]{MR3683503}
  that there cannot be a universal object
  as there is a Brauer obstruction
  (assuming the natural conjecture in Conjecture~4.3 of op.~cit.).
\end{remark}

As $\mathbf{a}$ will be fixed,
we often just write $\mathcal{U}=((\mathcal{U}_i)_{i\in Q_0},(\mathcal{U}_a)_{a\in Q_1})$.
It will be clear from the context whether we are considering~$\mathcal{U}$ on the moduli space or on a moduli stack.

Just like for representations of $Q$ over a field,
the category of representations of $Q$ in the category of quasicoherent $\mathcal{O}_S$-modules
for a $k$-scheme (or stack) $S$
is equivalent to the category of quasicoherent left modules over the sheaf of noncommutative algebras $\mathcal{O}_SQ$.
We will not make a distinction between the two.

\begin{remark}
  \label{remark:functor-of-points}
  Using the universal representation
  we are able to describe the functor of points of $\modulispace[\stable{\theta}]{Q,\mathbf{d}}$ concretely.

  Let $S$ be a scheme over $k$.
  We define an equivalence relation $\sim$ on left $\mathcal{O}_SQ$-modules
  by setting $\mathcal{M} \sim \mathcal{N}$ if there exists a line bundle $\mathcal{L}$
  such that $\mathcal{M} \otimes_{\mathcal{O}_S} \mathcal{L} \cong \mathcal{N}$ as left $\mathcal{O}_SQ$-modules.
  The set of $S$-valued points of $\modulispace[\stable{\theta}]{Q,\mathbf{d}}$ is
  \begin{equation}
    \begin{aligned}
      \big\{ \mathcal{M}
        &\mid
        \text{$\mathcal{M}$ left $\mathcal{O}_SQ$-module, }
        \text{$\forall i\in Q_0$: $\mathcal{M}_i$ locally free of rank $d_i$}\\
        &\qquad\text{$\mathcal{M}(p) \otimes_{\kappa(p)} K$ is $\theta$-stable for every $p \in S$ and every finite field extension $K/\kappa(p)$}
      \big\} / {\sim}.
    \end{aligned}
  \end{equation}
  In the above, $\kappa(p) = \mathcal{O}_{S,p}/\mathfrak{m}_p$ denotes
  the residue field of the (not necessarily closed) point $p \in S$
  and $\mathcal{M}(p) = \mathcal{M} \otimes_{\mathcal{O}_S} \kappa(p)$ is the fiber of $\mathcal{M}$;
  it is a representation of $Q$ over $\kappa(p)$.
  If $\mathcal{M}$ is a left $\mathcal{O}_SQ$-module
  such that all $\mathcal{M}_i$ are locally free of rank $d_i$
  and all finite extensions of all fibers are $\theta$-stable,
  then there exists a unique morphism $f\colon S \to \modulispace[\stable{\theta}]{Q,\mathbf{d}}$
  and a line bundle $\mathcal{L}$ on $S$ such that $\mathcal{M} \otimes_{\mathcal{O}_S} \mathcal{L} \cong f^*\mathcal{U}$.
\end{remark}

\subsection{Tautological presentation for the Chow ring}
In this section we recall the tautological presentation of
the Chow ring of a quiver moduli space.
We continue to work with \cref{assumption:standing}.
Let
\begin{equation}
  X = \modulispace[\stable\theta]{Q,\mathbf{d}};
\end{equation}
it is smooth and projective by \cref{proposition:basic-properties}

We fix~$\mathbf{a}$ such that~$\theta(\mathbf{a}) = 1$,
and we let~$\mathcal{U}$ be the universal representation which corresponds to this choice.
The following is \cite[Theorem~3]{MR1324213}:

\begin{theorem}[King--Walter]
  \label{theorem:king-walter}
  The Chow ring~$\chow^\bullet(X)$ is, as a~$\mathbb{Z}$-algebra,
  generated by the Chern classes~$\chern_r(\mathcal{U}_i)$
  of the summands of the universal representation.
\end{theorem}

As a consequence of this result,
we obtain the following.
\begin{corollary}
  \label{corollary:picard-group-presentation}
  We have that~$\Pic(X)\cong\chow^1(X)$ is generated as an abelian group
  by the first Chern classes~$\chern_1(\mathcal{U}_i)$.
  They are subject to the relation~$\sum_{i \in Q_0} a_i\chern_1(\mathcal{U}_i) = 0$.
  If the codimension of the unstable locus $\representationvariety{Q,\mathbf{d}} \setminus \representationvariety[\stable\theta]{Q,\mathbf{d}}$
  is at least 2,
  then this is the only relation among the~$\chern_1(\mathcal{U}_i)$.
\end{corollary}

Thus the normalization as discussed in \cref{corollary:descent}
plays an important role in the presentation of~$\Pic(X)$,
and likewise for all of~$\chow^\bullet(X)$.

We now recall the presentation of the Chow ring with rational coefficients from \cite{MR3318266},
which is expressed in the Chern roots of the universal representation.
We consider variables~$\xi_{i,k}$ for all~$i \in Q_0$ and~$k=1,\ldots,d_i$
and define~$x_{i,k} = \mathrm{e}_k(\xi_{i,1},\ldots,\xi_{i,d_i})$,
the $k$th elementary symmetric polynomial.
Let $\weyl_{\mathbf{d}}\colonequals\prod_{i\in Q_0}\mathrm{Sym}_{d_i}$.
Define
\begin{equation}
  \begin{aligned}
    R &\colonequals \bigotimes_{i\in Q_0}\mathbb{Q}[\xi_{i,1},\ldots,\xi_{i,d_i}] \\
    A &\colonequals R^{\weyl_{\mathbf{d}}}
    =\bigotimes_{i\in Q_0}\mathbb{Q}[x_{i,1},\ldots,x_{i,d_i}]
  \end{aligned}
\end{equation}
where~$\weyl_{\mathbf{d}}$ acts by permuting the factors.
We define the \emph{symmetrization map}
\begin{equation}
  \rho\colon R\to A:f\mapsto
  \frac{1}{\delta}\sum_{\sigma\in\weyl_{\mathbf{d}}}\operatorname{sign}(\sigma)\sigma\cdot f
\end{equation}
where~$\delta\colonequals\prod_{i\in Q_0}\prod_{1\leq k<\ell\leq d_i}(\xi_{i,\ell}-\xi_{i,k})$ is the discriminant.
It is a homomorphism of $A$-modules.
As an $A$-module, $R$ is free of rank~$\#\weyl_{\mathbf{d}}$.

We define~$I_\tautological$ as the ideal of~$R$ generated by the polynomials
\begin{equation}
  \label{equation:relation-from-forbidden}
  \prod_{a\in Q_1}\prod_{k=1}^{d_{\source(a)}'}\prod_{\ell=d_{\target(a)}'+1}^{d_{\target(a)}}\left( \xi_{\target(a),\ell}-\xi_{\source(a),k} \right)
\end{equation}
where~$0\lneqq\mathbf{d}'\lneqq\mathbf{d}$ such that~$\theta(\mathbf{d}')>0$,
the so-called \emph{forbidden} dimension vectors.
Then~$\rho(I_\tautological)$ is an ideal of~$A$.
In fact, it is enough to consider only minimal forbidden dimension vectors with respect to
the partial order ${\trianglelefteq}$ introduced in \cite[\S5.1]{MR3318266}:
when~$\mathbf{d}'\trianglelefteq\mathbf{d}''$ are both forbidden,
then the polynomial \eqref{equation:relation-from-forbidden} for~$\mathbf{d}'$
divides that for~$\mathbf{d}''$.

The \emph{linear relation} is the one arising as in \cref{corollary:picard-group-presentation}
and depends on the choice of~$\mathbf{a}$ such that~$\theta(\mathbf{a})=1$,
thus it is not tautological.
We define~$I_\linear$ as the principal ideal of~$A$ generated by
\begin{equation}
  \sum_{i\in Q_0}a_ix_{i,1}.
\end{equation}

The following result is \cite[Theorem~14]{MR3318266}.
\begin{theorem}[Franzen]
  \label{theorem:franzen}
  Assume \cref{assumption:standing}.
  Then the kernel of the surjective $\mathbb{Q}$-algebra homomorphism
  \begin{equation}
    \bigotimes_{i \in Q_0} \mathbb{Q}[x_{i,1},\ldots,x_{i,d_i}] \twoheadrightarrow \chow^\bullet_{\mathbb{Q}}(X):x_{i,k}\mapsto\chern_k(\mathcal{U}_i)
  \end{equation}
  is~$I_\linear + \rho(I_\tautological)$.
\end{theorem}

\section{The tautological and Kodaira--Spencer morphism}
In \cref{subsection:tautological-morphism}
we construct an important morphism of vector bundles
in the study of moduli of quiver representations.
The morphism is certainly not novel,
but we believe the details of its construction in this generality are.
What \emph{is} novel (at least to our knowledge)
is its relationship to the Kodaira--Spencer morphism for quiver moduli in \cref{subsection:kodaira-spencer}.

In \cref{subsection:points}
(resp.~\cref{subsection:todd-class})
we explain how the tautological morphism
(resp.~tautological and Kodaira--Spencer morphism)
can be used to prove \cref{theorem:point-class,theorem:todd-class}.

\subsection{The tautological morphism}
\label{subsection:tautological-morphism}
Let $M$ and $N$ be two finite-dimensional representations of $Q$ over $k$.
Then applying the functor $\Hom_{kQ}(-,N)$
to the standard projective resolution of $M$ yields the \emph{tautological exact sequence}
\begin{equation}
  \label{equation:tautological-sequence}
  0 \to \Hom_{kQ}(M,N) \to \bigoplus_{i \in Q_0} \Hom(M_i,N_i) \xrightarrow[]{\sigma_{M,N}} \bigoplus_{a \in Q_1} \Hom(M_{\source(a)},N_{\target(a)}) \to \Ext_{kQ}^1(M,N) \to 0.
\end{equation}
The map $\sigma_{M,N}$ is given by $\sigma_{M,N}((f_i)_{i \in Q_0}) = (f_{\target(a)}\circ M_a - N_a\circ f_{\source(a)})_{a \in Q_1}$.

The morphism $\sigma_{M,N}$ also exists when~$k$ is replaced by a commutative $k$-algebra~$A$ and $M$ and $N$ are left $AQ$-modules.
The construction of $\sigma_{M,N}$ is compatible with base change, i.e.,~for a homomorphism $A \to B$ of a commutative $k$-algebras holds
\begin{equation}
  \label{equation:base-change}
  \sigma_{M,N} \otimes_A \id = \sigma_{M \otimes_A B,N \otimes_A B}.
\end{equation}
Moreover, if $M$ is projective as an $A$-module then the left $AQ$-module has a standard projective resolution of length one, analogous to the case over a field. An application of $\Hom_{AQ}(-,N)$ yields $\ker \sigma_{M,N} = \Hom_{AQ}(M,N)$ and $\coker \sigma_{M,N} = \Ext_{AQ}^1(M,N)$.

This in fact allows us to construct $\sigma$ in great generality over a scheme,
and even algebraic stack.
Let $S$ be an algebraic stack over $\Spec k$,
and let $\mathcal{M}$ and $\mathcal{N}$ be two representations of $Q$
in the category of quasicoherent sheaves on~$S$;
we call them quasicoherent left $\mathcal{O}_SQ$-modules.
We define a homomorphism of sheaves
\begin{equation}
  \label{equation:tautological-morphism-in-general}
  \sigma = \sigma_{\mathcal{M},\mathcal{N}}\colon
  \bigoplus_{i \in Q_0} \sheafHom_{\mathcal{O}_S}(\mathcal{M}_i,\mathcal{N}_i)
  \to
  \bigoplus_{a \in Q_1} \sheafHom_{\mathcal{O}_S}(\mathcal{M}_{\source(a)}, \mathcal{N}_{\target(a)})
\end{equation}
by mapping a section $(f_i)_{i \in Q_0}$ on $T\to S$
to~$(f_{\target(a)}\circ \mathcal{M}_a|_T - \mathcal{N}_a|_T\circ f_{\source(a)})_{a \in Q_1}$.

We define $\sheafHom_{\mathcal{O}_SQ}(\mathcal{M},\mathcal{N})$ as the kernel of $\sigma_{\mathcal{M},\mathcal{N}}$.
It is a quasicoherent $\mathcal{O}_S$-module whose sections over an open subset $U \subseteq S$ are
\begin{equation}
  \label{equation:sections-relative-Hom}
  \HH^0(U,\sheafHom_{\mathcal{O}_SQ}(\mathcal{M},\mathcal{N})) = \Hom_{\mathcal{O}_UQ}(\mathcal{M}|_U,\mathcal{N}|_U).
\end{equation}
It is easy to see that the functor $\sheafHom_{\mathcal{O}_SQ}(\mathcal{M},-)$
from quasicoherent left $\mathcal{O}_SQ$-modules to quasicoherent $\mathcal{O}_S$-modules is left exact.
Moreover, the category of quasicoherent left $\mathcal{O}_SQ$-modules has enough injectives.

\begin{definition}
  \label{definition:relative-Ext}
  Let $\mathcal{M}$ and $\mathcal{N}$ be two quasicoherent left $\mathcal{O}_SQ$-modules. We define
  \begin{equation}
    \sheafExt_{\mathcal{O}_SQ}^n(\mathcal{M},\mathcal{N})
    \colonequals
    (\mathrm{R}^n\sheafHom_{\mathcal{O}_SQ}(\mathcal{M},-))(\mathcal{N}).
  \end{equation}
\end{definition}

The functor $\sheafExt_{\mathcal{O}_SQ}^n$ is related to global extensions through a local-to-global spectral sequence
\begin{equation}
  \label{equation:local-to-global-ss}
  \mathrm{E}_2^{p,q} = \HH^p(S,\sheafExt_{\mathcal{O}_SQ}^q(\mathcal{M},\mathcal{N}))
  \Rightarrow \Ext_{\mathcal{O}_SQ}^{p+q}(\mathcal{M},\mathcal{N}).
\end{equation}

\begin{lemma}
  \label{lemma:balancing-like-statement}
  Let $\mathcal{M}$ be a quasicoherent left $\mathcal{O}_SQ$-module and suppose that it admits a resolution
  \begin{equation}
    \ldots \to \mathcal{E}_2 \to \mathcal{E}_1 \to \mathcal{E}_0 \to \mathcal{M} \to 0
  \end{equation}
  by quasicoherent left $\mathcal{O}_SQ$-modules $\mathcal{E}_n$
  which are acyclic for $\sheafHom_{\mathcal{O}_SQ}(-,\mathcal{N})$ for every quasicoherent left $\mathcal{O}_SQ$-module $\mathcal{N}$.
  Then there exist isomorphisms
  \begin{equation}
    \sheafExt_{\mathcal{O}_SQ}^n(\mathcal{M},\mathcal{N}) \cong \mathbb{H}^n(\sheafHom_{\mathcal{O}_SQ}(\mathcal{E}_*,\mathcal{N})).
  \end{equation}
  which are natural in $\mathcal{N}$.
\end{lemma}

\begin{proof}
  We fix~$\mathcal{M}$ as in the statement.
  Let $T \colonequals (T^n)_{n \geq 0}$
  be given by $T^n(\mathcal{N})\colonequals \Ext^n(\mathcal{M},\mathcal{N})$
  and let $\tilde{T} \colonequals (\tilde{T}^n)_{n \geq 0}$
  be given by $\tilde{T}^n(\mathcal{N}) \colonequals \mathbb{H}^n(\sheafHom_{\mathcal{O}_SQ}(\mathcal{E}_*,\mathcal{N}))$.
  The collection of functors $T$ is a universal $\delta$-functor by definition.

  As the complex $\mathcal{E}_*$ consists of $\sheafHom_{\mathcal{O}_SQ}(-,\mathcal{N})$-acyclic objects,
  the collection $\tilde{T}$ is a $\delta$-functor.
  It is universal because for every $n > 0$,
  the functor $\tilde{T}^n$ vanishes on injective $\mathcal{O}_SQ$-modules.

  Since the functors $T^0$ and $\tilde{T}^0$ both agree with $\sheafHom_{\mathcal{O}_SQ}(\mathcal{M},-)$,
  we deduce that there exists a unique isomorphism~$T \to \tilde{T}$ of $\delta$-functors which extends the isomorphism $T^0 \to \tilde{T}^0$.
\end{proof}

We apply \cref{lemma:balancing-like-statement} to the following situation.
Suppose that $\mathcal{M}$ is locally free of finite rank as $\mathcal{O}_S$-module.
Consider the indecomposable projective $kQ$-module $P(j)$
and a left $\mathcal{O}_SQ$-module of the form $P(j) \otimes \mathcal{E}$
where $\mathcal{E}$ is a locally free $\mathcal{O}_S$-module of finite rank.
It satisfies
\begin{equation}
  \sheafHom_{\mathcal{O}_SQ}(P(j)\otimes\mathcal{E},\mathcal{N}) = \sheafHom_{\mathcal{O}_S}(\mathcal{E},\mathcal{N}_j)
\end{equation}
and thus it is acyclic for the functor $\sheafHom_{\mathcal{O}_SQ}(-,\mathcal{N})$.
We consider the resolution
\begin{equation}
  0 \to \bigoplus_{a \in Q_1} P(\target(a)) \otimes \mathcal{M}_{\source(a)} \to \bigoplus_{i \in Q_0} P(i) \otimes \mathcal{M}_i \to \mathcal{M} \to 0
\end{equation}
by acyclic objects.
Applying the functor~$\sheafHom_{\mathcal{O}_SQ}(-,\mathcal{N})$ then yields the global version of \eqref{equation:tautological-sequence}.

\begin{proposition}
  \label{proposition:tautological-sequence-sheafy}
  Let $\mathcal{M}$ and $\mathcal{N}$ be quasicoherent left $\mathcal{O}_SQ$-modules
  and assume that $\mathcal{M}$ is locally free of finite rank as an $\mathcal{O}_S$-module.
  Then there exists an exact sequence
  \begin{equation}
    0
    \to \sheafHom_{\mathcal{O}_SQ}(\mathcal{M},\mathcal{N})
    \to \bigoplus_{i \in Q_0} \sheafHom_{\mathcal{O}_S}(\mathcal{M}_i,\mathcal{N}_i)
    \xrightarrow[]{\sigma_{\mathcal{M},\mathcal{N}}} \bigoplus_{a \in Q_1} \sheafHom_{\mathcal{O}_S}(\mathcal{M}_{\source(a)},\mathcal{N}_{\target(a)})
    \to \sheafExt_{\mathcal{O}_SQ}^1(\mathcal{M},\mathcal{N})
    \to 0.
  \end{equation}
  and $\sheafExt_{\mathcal{O}_SQ}^n(\mathcal{M},\mathcal{N}) = 0$ for all $n \geq 2$.
\end{proposition}

\begin{lemma}
  \label{lemma:fibers-hom-ext}
  In the situation of \cref{proposition:tautological-sequence-sheafy},
  assume that $S$ is a $k$-scheme and
  let~$p \in S$ be a closed point.
  Then
  \begin{equation}
    \begin{aligned}
      \ker(\sigma(p)) &= \Hom_{\kappa(p)Q}(\mathcal{M}(p),\mathcal{N}(p)) \\
      \coker(\sigma(p)) &= \Ext^1_{\kappa(p)Q}(\mathcal{M}(p),\mathcal{N}(p))
    \end{aligned}
  \end{equation}
\end{lemma}

\begin{proof}
  This follows from \cref{proposition:tautological-sequence-sheafy} and \eqref{equation:base-change}.
\end{proof}

For the proofs of \cref{proposition:diagonal,theorem:todd-class}
we will need the following result.
\begin{lemma}
  \label{lemma:sublinebundle}
  Let $S$ be a $k$-scheme of finite type
  and let $\mathcal{M}$ and $\mathcal{N}$ be two left $\mathcal{O}_SQ$-modules,
  locally free of finite rank as $\mathcal{O}_S$-modules, such that
  \begin{equation}
    \dim \Hom_{kQ}(\mathcal{M}(p),\mathcal{N}(p)) = 1
  \end{equation}
  for all $k$-valued points $p$ of $S$.
  Then the kernel of $\sigma$ is a locally split line subbundle of $\bigoplus_{i \in Q_0} \mathcal{M}_i^\vee \otimes \mathcal{N}_i$.
\end{lemma}

\begin{proof}
  Denote~$e = \rk \bigoplus_{i \in Q_0} \mathcal{M}_i^\vee \otimes \mathcal{N}_i$.
  The assumption that $\Hom_{kQ}(\mathcal{M}(p),\mathcal{N}(p))$ is one-dimensional
  ensures that~$\bigwedge^e \sigma = 0$
  and~$\bigwedge^{e-1} \sigma$ vanishes nowhere.
  In other words, top degeneracy loci are~$\degeneracy_{e-1}(f)=S$ and~$\degeneracy_{e-2}(f)=\emptyset$,
  as these degeneracy loci can be described as the vanishing loci of~$\bigwedge^e\sigma$ resp.~$\bigwedge^{e-1}\sigma$.
  Thus the rank of~$\sigma$ is constant~$e-1$ at all points~$s\in S$,
  and~$\ker\sigma$ is then a locally split line subbundle
  by the usual argument \cite[\href{https://stacks.math.columbia.edu/tag/0FWH}{Tag 0FWH}]{stacks-project}.
\end{proof}

\begin{example}
  \label{example:two-settings}
  This construction of~$\sigma$ may be applied to $S = X \times Y$, where
  \begin{itemize}
    \item $X$ and $Y$ are both moduli stacks $X = \modulistack{Q,\mathbf{d}}$
      and $Y = \modulistack{Q,\mathbf{e}}$,
    \item $X$ and $Y$ are both moduli spaces of stable representations $X = \modulispace[\stable{\theta}]{Q,\mathbf{d}}$
      and $Y = \modulispace[\stable{\eta}]{Q,\mathbf{e}}$,
      with $\mathbf{d}$ and~$\mathbf{e}$ both indivisible.
  \end{itemize}
  In both cases,
  let $\mathcal{U}$ be the universal representation of $X$
  and $\mathcal{V}$ the universal representation of $Y$.
  We consider~$\mathcal{M} = p_1^*\mathcal{U}$ and $\mathcal{N} = p_2^*\mathcal{V}$
  and obtain the exact sequence
  \begin{equation}
    \label{equation:tautological-morphism-over-product}
    0
    \to \sheafHom_{\mathcal{O}_SQ}(p_1^*\mathcal{U},p_2^*\mathcal{V})
    \to \bigoplus_{i \in Q_0} \mathcal{U}_i^\vee \boxtimes \mathcal{V}_i
    \xrightarrow[]{\sigma} \bigoplus_{a \in Q_1} \mathcal{U}_{\source(a)}^\vee \boxtimes \mathcal{V}_{\target(a)}
    \to \sheafExt_{\mathcal{O}_SQ}^1(p_1^*\mathcal{U},p_2^*\mathcal{V})
    \to 0.
  \end{equation}
  We will in fact only ever need~$X=Y$ in what follows.
\end{example}

\subsection{The Kodaira--Spencer morphism via Atiyah classes}
\label{subsection:kodaira-spencer}
In order to describe the Todd class in \cref{theorem:todd-class}
we need to identify the cokernel of the restriction of~$\sigma$ to the diagonal
with the tangent bundle.
There exists an equivariant proof for this
(which we will also give, see the second proof of \cref{lemma:4-term-tangent-sequence})
but the following proposition is more intrinsic,
and more general,
and uses the moduli-theoretic interpretation.
\begin{proposition}
  \label{proposition:kodaira-spencer}
  Let $S$ be a finite-type $k$-scheme
  and let $\mathcal{M}$ be a coherent left $\mathcal{O}_SQ$-module which
  is locally free as an~$\mathcal{O}_S$-module.
  \begin{enumerate}
    \item
      \label{enumerate:kodaira-spencer-morphism}
      There exists a Kodaira--Spencer type morphism
      \begin{equation}
        \label{equation:kodaira-spencer}
        \kodairaspencer_{S,\mathcal{M}}\colon \tangent_S \to \sheafExt_{\mathcal{O}_SQ}^1(\mathcal{M},\mathcal{M}).
      \end{equation}

    \item
      \label{enumerate:kodaira-spencer-isomorphism}
      If $X=\modulispace[\stable{\theta}]{Q,\mathbf{d}}$ for an acyclic quiver
      and a $\theta$-coprime dimension vector,
      and $\mathcal{M} = \mathcal{U}$ is the universal representation,
      then
      \begin{equation}
        \label{equation:kodaira-spencer-isomorphism}
        \kodairaspencer=\kodairaspencer_{X,\mathcal{U}}\colon\tangent_X\to\sheafExt_{\mathcal{O}_XQ}(\mathcal{U},\mathcal{U})
      \end{equation}
      is an isomorphism.
  \end{enumerate}
\end{proposition}
This is the analogue of \cite[Theorem~10.2.1]{MR2665168},
which is set in the context of moduli spaces of sheaves on surfaces.
In \cite[Lemma~2.6]{MR0325615} a version for vector bundles on a curve is given.
Recall that in these geometric settings
the infinitesimal Kodaira--Spencer morphism for a flat family~$\mathcal{V}$ of objects on a variety~$X$
over some base~$B$,
is the morphism
\begin{equation}
  \tangent_0B\to\Ext_X^1(V,V),
\end{equation}
where~$V$ is the fiber of~$\mathcal{V}$ at the base point~$0\in B$.
It can be constructed in families,
over a smooth base~$B$,
to give the Kodaira--Spencer morphism
\begin{equation}
  \tangent_B\to\sheafExt_p^1(\mathcal{V},\mathcal{V}),
\end{equation}
as discussed, e.g., in \cite[\S10.1]{MR2665168},
where~$\sheafExt_p^1(\mathcal{V},-)$ denotes the first derived functor of
the left exact composition~$p_*\circ\sheafHom_{S\times X}(\mathcal{V},-)$
with~$p\colon S\times X\to S$ the first projection.
This is the geometric incarnation of the functor defined in \cref{definition:relative-Ext}:
in the noncommutative setup the projection map is already part of the functor~$\sheafHom_{\mathcal{O}_SQ}$
as this produces only an~$\mathcal{O}_S$-module.

There are several ways ways to construct \eqref{equation:kodaira-spencer}.
In \cite[\S2.4]{MR3925499} a construction using trivialisations is sketched.
We will follow the approach in \cite[\S10.1]{MR2665168} instead,
explaining how Atiyah classes enter the picture in this context.

\begin{lemma}
  \label{lemma:atiyah-class}
  Let~$S$ be a scheme of finite type over~$k$.
  Let~$\mathcal{M}$ be an~$\mathcal{O}_SQ$-module
  which is locally free as an~$\mathcal{O}_S$-module.
  The universal Atiyah class on~$S\times S$ defines a class
  \begin{equation}
    \label{equation:atiyah-class}
    \at_{S,\mathcal{M}}\in\Ext_{\mathcal{O}_SQ}^1(\mathcal{M},\mathcal{M}\otimes\Omega_S^1).
  \end{equation}
\end{lemma}

\begin{proof}
  The universal Atiyah class on~$S\times S$ is the extension class
  \begin{equation}
    \At_S\in\Ext_{S\times S}^1(\Delta_*\mathcal{O}_S,\Delta_*\Omega_S^1)
  \end{equation}
  defined by the short exact sequence~$0\to\mathcal{I}/\mathcal{I}^2\to\mathcal{O}_{S\times S}/\mathcal{I}^2\to\Delta_*\mathcal{O}_S\to 0$
  where~$\mathcal{I}/\mathcal{I}^2\cong\Delta_*\Omega_S^1$,
  and~$\Delta\colon S\hookrightarrow S\times S$ is the diagonal inclusion.
  The universal Atiyah class thus gives us a natural transformation of the associated Fourier--Mukai transforms,
  i.e., between the identity functor and~$-\otimes^{\mathbf{L}}\Omega_S^1[1]$.

  We can lift the Fourier--Mukai transforms and their natural transformation
  along the forgetful functor from~$\mathcal{O}_SQ$-modules to~$\mathcal{O}_S$-modules,
  i.e., the image under a Fourier--Mukai transform of an object with an~$\mathcal{O}_SQ$-structure
  has an~$\mathcal{O}_SQ$-structure.
  Applying this reasoning to the object~$\mathcal{M}$
  with the natural transformation induced by~$\At_S$
  we obtain the extension class
  \begin{equation}
    \at_{S,\mathcal{M}}\in\Ext_{\mathcal{O}_SQ}^1(\mathcal{M},\mathcal{M}\otimes\Omega_S^1)
  \end{equation}
  as in the statement.
\end{proof}

Observe that in the geometric setting of \cite[\S10.1]{MR2665168},
one needs to consider a projection onto the base
to define the appropriate Atiyah class.
This does not play a role in our setting,
as explained in the discussion before \cref{lemma:atiyah-class}.

\begin{proof}[Proof of \cref{proposition:kodaira-spencer}]
  First we discuss the construction of the morphism \eqref{equation:kodaira-spencer}.
  From \eqref{equation:local-to-global-ss}
  we obtain the morphism
  \begin{equation}
    \Ext_{\mathcal{O}_SQ}^1(\mathcal{M},\mathcal{M}\otimes\Omega_S^1)
    \to
    \HH^0(S,\sheafExt_{\mathcal{O}_SQ}(\mathcal{M},\mathcal{M}\otimes\Omega_S^1)).
  \end{equation}
  The image of the Atiyah class \eqref{equation:atiyah-class}
  thus induces a morphism
  \begin{equation}
    \mathcal{O}_S\to\sheafExt_{\mathcal{O}_SQ}^1(\mathcal{M},\mathcal{M}\otimes\Omega_S^1).
  \end{equation}
  Tensoring this morphism with the dual of~$\Omega_S^1$,
  and composing with the obvious morphisms, we obtain
  \begin{equation}
    \kodairaspencer_{S,\mathcal{M}}\colon
    \tangent_S
    \to
    \tangent_S\otimes\sheafExt_{\mathcal{O}_SQ}^1(\mathcal{M},\mathcal{M}\otimes\Omega_S^1)
    \to
    \sheafExt_{\mathcal{O}_SQ}^1(\mathcal{M},\mathcal{M}\otimes\sheafEnd(\Omega_S^1))
    \to
    \sheafExt_{\mathcal{O}_SQ}^1(\mathcal{M},\mathcal{M})
  \end{equation}
  which is the desired morphism \eqref{equation:kodaira-spencer}.

  Now we can prove that \eqref{equation:kodaira-spencer-isomorphism} is an isomorphism,
  by checking it fiberwise.
  For this we follow the approach in \cite[Example~10.1.9]{MR2665168},
  which is standard in the geometric setting,
  but not spelled out in this noncommutative setting as far as we know.

  Let~$x\in X$ be a closed point, corresponding to a (stable) representation~$U$.
  Let~$v\in\Ext_{kQ}^1(U,U)$ be an extension class,
  which is an element in the fiber of~$\sheafExt_{\mathcal{O}_XQ}^1(\mathcal{U},\mathcal{U})$
  at the point~$x$.

  The extension class~$v$ corresponds to a short exact sequence
  \begin{equation}
    \label{equation:short-exact-sequence-extension}
    0\to U\to V\to U\to 0
  \end{equation}
  of~$kQ$-modules.
  The composition of projection and inclusion gives $V$ the structure of a $k[\epsilon]/(\epsilon^2)$-module.
  As a~$k[\epsilon]/(\epsilon^2)$-module $V$ is flat.

  Denote~$R=k[\epsilon]/(\epsilon^2)$
  and take~$S=\Spec R$ with~$\mathcal{M}=V$,
  so that \cref{proposition:kodaira-spencer} (\ref{enumerate:kodaira-spencer-morphism})
  gives
  \begin{equation}
    \label{equation:at-SV-identification}
    \at_{S,V}\in\Ext_{RQ}^1(V,V\otimes\Omega_{R/k}^1)\cong\Ext_{kQ}^1(U,U)
  \end{equation}
  by flatness and the isomorphism~$\Omega_{R/k}^1\cong k\cdot\mathrm{d}\epsilon$.
  As the Kodaira--Spencer map is induced by the universal Atiyah class
  it suffices to show that~$\at_{S,V}=v$ under these identifications:
  the fiberwise Kodaira--Spencer map is then a surjection~$\Ext_{kQ}^1(U,U)\to\Ext_{kQ}^1(U,U)$,
  and thus an isomorphism.

  On~$S\times S=\Spec R\otimes_kR$
  with~$R\otimes_kR=k[\epsilon_1,\epsilon_2]/(\epsilon_1,\epsilon_2)^2$
  the natural transformation
  attached to the universal Atiyah class
  gives a short exact sequence
  \begin{equation}
    \label{equation:short-exact-sequence-on-SxS}
    0\to U\to W\to V\to 0
  \end{equation}
  of~$(R\otimes_kR)Q$-representations,
  where~$\epsilon_i$ acts by~0 on~$U$ for~$i=1,2$,
  and acts by~$\epsilon$ on~$V$ for~$i=1,2$.
  Now~$W\cong U\oplus V$ over~$k$,
  with~$\epsilon_1$ acting by the projection~$V\to U$
  and~$\epsilon_2$ acting by~$\epsilon$ on~$V$.
  We can now conclude as in \cite[Example~10.1.9]{MR2665168}:
  consider the extension \eqref{equation:short-exact-sequence-on-SxS}
  as an element of the left-hand side of the isomorphism in \eqref{equation:at-SV-identification}
  by taking~$\epsilon=\epsilon_1$.
  We have the pullback diagram
  \begin{equation}
    \begin{tikzcd}
      0 \arrow[r] & U \arrow[d, equals] \arrow[r] & V \arrow[r] & U \arrow[r] & 0 \\
      0 \arrow[r] & U \arrow[r] & W \arrow[r] \arrow[u] & V \arrow[r] \arrow[u] & 0
    \end{tikzcd}
  \end{equation}
  where all the maps are as in \eqref{equation:short-exact-sequence-extension} and \eqref{equation:short-exact-sequence-on-SxS},
  with the top row corresponding to~$v$ and the bottom row to~$\at_{S,V}$ (before the isomorphism in \eqref{equation:at-SV-identification}),
  and thus we obtain (after the isomorphism in \eqref{equation:at-SV-identification}) the equality~$v=\at_{S,V}$.
\end{proof}

\section{Identities in the Chow ring}
We are now ready to prove \cref{theorem:point-class,theorem:todd-class}.

\subsection{Points as degeneracy loci}
\label{subsection:points}
The following proposition is the quiver moduli analog of \cite[Lemma~2.4]{MR1228610}.
We use results on degeneracy loci from \cite[\S14.4]{MR1644323}.
\begin{proposition}
  \label{proposition:diagonal}
  Let~$X$ be~$\modulispace[\theta]{Q,\mathbf{d}}$.
  With~$\sigma$ on $X \times X$ as in \cref{example:two-settings} we have that
  \begin{equation}
    \Delta_X=\degeneracy_{e-1}(\sigma)\subset X\times X,
  \end{equation}
  where~$e=\sum_{i\in Q_0}d_i^2=\rk\Big(\bigoplus_{i\in Q_0}\mathcal{U}_i^\vee\boxtimes\mathcal{U}_i\Big)$.
  Moreover, the above equation holds scheme-theoretically,
  and $\degeneracy_{e-1}(\sigma)$ is of the expected codimension.
\end{proposition}

\begin{proof}
  By \cref{corollary:smooth-projective-quiver-moduli}
  every semistable representation is stable,
  so by Schur's lemma and \cref{lemma:fibers-hom-ext} we have that
  \begin{equation}
    \Hom_{kQ}(M,N)\neq 0\Leftrightarrow M\cong N.
  \end{equation}
  Thus~$\degeneracy_{e-1}(\sigma)$ cuts out the diagonal.
  The expected codimension of $\degeneracy_{e-1}(\sigma)$ is
  \begin{equation}
    \rk\Big( \bigoplus_{a \in Q_1} \mathcal{U}_{\source(a)}^\vee \boxtimes \mathcal{U}_{\target(a)} \Big) - \rk\Big( \bigoplus_{i\in Q_0}\mathcal{U}_i^\vee\boxtimes\mathcal{U}_i \Big) + 1 = 1 - \langle \mathbf{d},\mathbf{d} \rangle = \dim X.
  \end{equation}
  We have now shown that $\Delta_X = \degeneracy_{e-1}(\sigma)_{\operatorname{red}}$.
  To further show that $\degeneracy_{e-1}(\sigma)$ is reduced,
  we consider the inclusion~$\iota\colon\degeneracy_{e-1}(\sigma) \to X \times X$.
  By \cref{lemma:sublinebundle}, the kernel of $\iota^*\sigma$ is
  a locally split line subbundle which we call $\mathcal{L}$.
  The inclusion provides us with a nowhere-vanishing section of
  \begin{equation}
    \mathcal{L}^\vee \otimes \iota^*\Big( \bigoplus_{i\in Q_0}\mathcal{U}_i^\vee\boxtimes\mathcal{U}_i \Big)
    =
    \bigoplus_{i\in Q_0}((\iota^*\circ p_1^*)(\mathcal{U}_i) \otimes \mathcal{L})^\vee \otimes(\iota^*\circ p_2^*)(\mathcal{U}_i)
  \end{equation}
  which corresponds to a family $f = (f_i)_{i \in Q_0}$ of
  isomorphisms $f_i\colon (\iota^*\circ p_1^*)(\mathcal{U}_i) \otimes \mathcal{L} \to (\iota^*\circ p_2^*)(\mathcal{U}_i)$.
  As $\mathcal{L}$ is the kernel of $\sigma$,
  these isomorphisms are compatible with the morphisms $(\iota^*\circ p_{\nu}^*)(\mathcal{U}_a)$ with $\nu = 1,2$
  and therefore yield an isomorphism of $\mathcal{O}_{\degeneracy_{e-1}(\sigma)}Q$-modules
  \begin{equation}
    f\colon(\iota^*\circ p_1^*)(\mathcal{U}) \otimes \mathcal{L} \to(\iota^*\circ p_2^*)(\mathcal{U}).
  \end{equation}
  The representations $(\iota^*\circ p_1^*)(\mathcal{U})$
  and $(\iota^*\circ p_2^*)(\mathcal{U})$ are therefore equivalent in the sense of \cref{remark:functor-of-points} which shows that $p_1\circ\iota = p_2\circ\iota$.
  This proves that $\iota$ factors through $\Delta_X$.
  The degeneracy locus $\degeneracy_{e-1}(\sigma)$ must hence be reduced.
\end{proof}

Fix~$x_0\in X$, corresponding to a~$\theta$-stable representation~$N$,
and denote
\begin{equation}
  \begin{aligned}
    j_1&\colon X\hookrightarrow X\times X:x\mapsto (x,x_0) \\
    j_2&\colon X\hookrightarrow X\times X:x\mapsto (x_0,x)
  \end{aligned}
\end{equation}
the sections to the projections~$p_1,p_2\colon X\times X\to X$.
The proof of \cref{theorem:point-class} now comes as a corollary of \cref{proposition:diagonal}.

\begin{proof}[Proof of \cref{theorem:point-class}]
  The proof of \cref{enumerate:point-class-1} uses that
  the formation of zero loci commutes with pullbacks,
  which can be proven by the standard argument for fiber product diagrams.
  As degeneracy loci are zero loci this holds for degeneracy loci as well.
  So $\degeneracy_{e-1}(j_1^*\sigma) = j_1^{-1}\degeneracy_{e-1}(\sigma)$, the scheme-theoretic inverse image.
  The latter is the inverse image of $\Delta_X$ under $j_1$ which is the point $\{x_0\}$ with the reduced subscheme structure.
  The proof for~$j_2$ is the same.

  The proof of \cref{enumerate:point-class-2} then follows from
  the Thom--Porteous formula \cite[Theorem~14.4]{MR1644323}
  and the special case of Example~14.4.1 in op.~cit.
\end{proof}

\subsection{The Todd class}
\label{subsection:todd-class}
To realise the full strength of computations in the Chow ring (or cohomology ring)
we want to apply the Hirzebruch--Riemann--Roch theorem
which says that for a vector bundle~$\mathcal{E}$ on~$X$
\begin{equation}
  \chi(X,\mathcal{E})=\int_X\Chern(\mathcal{E})\todd_X
\end{equation}
where~$\Chern(\mathcal{E})$ is Chern character
and~$\todd_X$ is the Todd class.
An expression for~$\todd_X$ in the generators of the Chow ring from \cref{theorem:king-walter}
is precisely the content of \cref{theorem:todd-class},
which we will now prove.
To this end, we consider the morphism $\sigma$
from \eqref{equation:tautological-morphism-over-product} for $X = Y$.

\begin{lemma}
  \label{lemma:4-term-tangent-sequence}
  Pulling back $\sigma$
  along the diagonal inclusion~$\Delta\colon X\hookrightarrow X\times X$
  we obtain the 4-term exact sequence
  \begin{equation}
    \label{equation:4-term-tangent-sequence}
    0
    \to\mathcal{O}_X
    \to\bigoplus_{i\in Q_0}\mathcal{U}_i^\vee\otimes\mathcal{U}_i
    \to\bigoplus_{a\in Q_1}\mathcal{U}_{\source(a)}^\vee\otimes\mathcal{U}_{\target(a)}
    \to\tangent_X
    \to0
  \end{equation}
\end{lemma}

We will give two proofs for this lemma.
The first uses the interpretation as a moduli space and the Kodaira--Spencer morphism,
the second uses the construction as a quotient
and is given as a more in-depth version of \cite[Remark~4.1]{MR4352662}.

\begin{proof}[Proof using moduli-theoretic properties]
  By \cref{lemma:sublinebundle} the kernel of $\Delta^*\sigma$
  is a locally split line subbundle.
  But the identical maps $\mathcal{U}_i \to \mathcal{U}_i$
  give a section of $\bigoplus_{i\in Q_0}\mathcal{U}_i^\vee\otimes\mathcal{U}_i$
  which vanishes nowhere,
  thus we obtain an injective morphism $\mathcal{O}_X \to \bigoplus_{i\in Q_0}\mathcal{U}_i^\vee\otimes\mathcal{U}_i$.
  As the identity on $\mathcal{U}$ is a homomorphism of $\mathcal{O}_X$-modules,
  the image of $\mathcal{O}_X \to \bigoplus_{i\in Q_0}\mathcal{U}_i^\vee\otimes\mathcal{U}_i$ lies in the kernel of $\Delta^*\sigma$. Therefore it is the kernel.

  The identification of the cokernel follows from the sheafy tautological sequence
  in \cref{proposition:tautological-sequence-sheafy}
  and the Kodaira--Spencer isomorphism from \cref{proposition:kodaira-spencer}.
\end{proof}

\begin{proof}[Proof using quotient description]
  If we pull back $\Delta^*\sigma$ along
  the quotient~$\pi\colon\representationvariety[\stable{\theta}]{Q,\mathbf{d}} \to \modulispace[\stable{\theta}]{Q,\mathbf{d}}$,
  we obtain the homomorphism
  \begin{equation}
    \alpha\colon\bigoplus_{i \in Q_0} \Hom(U_i,U_i) \to \bigoplus_{a \in Q_1} \Hom(U_{\source(a)},U_{\target(a)})
  \end{equation}
  which, in the fiber over a point $M = (M_a)_{a \in Q_1}$
  of~$\representationvariety[\stable{\theta}]{Q,\mathbf{d}}$,
  maps $f = (f_i)_{i \in Q_0}$ to $(f_{\target(a)}\circ M_a - M_a\circ f_{\source(a)})_{a \in Q_1}$.

  We are now in the setup for the generalized Euler sequence from \cite[\S5.1]{MR3126932}:
  by our assumption on~$\mathbf{d}$ and~$\theta$
  the semistable locus equals the stable locus,
  and the group~$\PG{\mathbf{d}}$ acts freely for quiver moduli.

  The domain of $\alpha$ identifies with the trivial bundle
  with fiber $\mathfrak{g}_{\mathbf{d}}$,
  the Lie algebra of the group $\group{\mathbf{d}}$,
  equipped with the structure of a $\group{\mathbf{d}}$-equivariant bundle
  by the adjoint action.

  The codomain of $\alpha$ is
  the trivial bundle with fiber $\representationvariety{Q,\mathbf{d}}$.
  It becomes a $\group{\mathbf{d}}$-equivariant bundle by
  the usual action and identifies with
  the tangent bundle of $\representationvariety[\stable{\theta}]{Q,\mathbf{d}}$.
  Thus we can descend the sequence to~$X$,
  and by \cite[(5.1.1)]{MR3126932}
  the cokernel of the descent of $\alpha$ is $\tangent_X$.

  The only difference with \cite[\S5.1]{MR3126932} is that~$\alpha$ is not injective.
  In the fiber over $M$,
  the homomorphism $\alpha$ is the derivative at $\id \in \group{\mathbf{d}}$ of the action map
  \begin{equation}
    \group{\mathbf{d}} \to \representationvariety[\stable{\theta}]{Q,\mathbf{d}},\ g \mapsto g \cdot M.
  \end{equation}
  The stabilizer of a stable representation $M$ is $\ker(\group{\mathbf{d}} \to \PG{\mathbf{d}})$
  which is a central subgroup of $\group{\mathbf{d}}$,
  isomorphic to~$\Gm$.
  Therefore the kernel of~$\alpha$ is the trivial bundle of rank one
  with the trivial $\group{d}$-linearizeation,
  and the 3-term generalized Euler sequence can equivalently be written
  as the 4-term sequence in \eqref{equation:4-term-tangent-sequence}.
\end{proof}

Using this lemma the proof of \cref{theorem:todd-class} is very short,
merely applying the usual properties of Todd classes of vector bundles.

\begin{proof}[Proof of \cref{theorem:todd-class}]
  The Todd class is multiplicative in short exact sequences,
  and in particular for direct sums.
  The Todd class of the last term in the~4-term sequence \eqref{equation:4-term-tangent-sequence}
  equals the alternating product of the other terms,
  hence we obtain \eqref{equation:todd-class},
  as~$\todd(\mathcal{O}_X)=1$.
\end{proof}

\begin{remark}
  \label{remark:explicit-todd-class}
  The Todd class of a vector bundle~$\mathcal{E}$ is defined using the splitting principle,
  as
  \begin{equation}
    \todd(\mathcal{E})\colonequals\prod_{i=1}^{\rk\mathcal{E}}\mathrm{Q}(\xi_i)
  \end{equation}
  where~$\xi_i$ are the Chern roots,
  and
  \begin{equation}
    \mathrm{Q}(t)\colonequals\frac{t}{1-\mathrm{e}^{-t}}
    =\sum_{k=0}^{+\infty}\frac{\mathrm{B}_k}{k!}t^k
  \end{equation}
  where~$\mathrm{B}_k$ is the~$k$th Bernoulli number.
  Using this we can rewrite \eqref{equation:todd-class} as
  \begin{equation}
    \todd_X
    =\frac{\displaystyle\prod_{a\in Q_1}\prod_{r=1}^{d_{\source(a)}}\prod_{s=1}^{d_{\target(a)}}\mathrm{Q}(\xi_{\target(a),s}-\xi_{\source(a),r})}{\displaystyle\prod_{i\in Q_0}\prod_{r=1}^{d_i}\prod_{s=1}^{d_i}\mathrm{Q}(\xi_{i,s}-\xi_{i,r})}.
  \end{equation}
\end{remark}

\section{On Kronecker moduli}
To illustrate the results from \cref{theorem:point-class,theorem:todd-class}
we will illustrate them for the~4~smallest ``interesting'' Kronecker moduli.

A natural quiver to consider is the $m$-Kronecker quiver
\begin{equation}
  \mathrm{K}_m\colon
  \begin{tikzpicture}[baseline={([yshift=-.5ex]current bounding box.center)},vertex/.style={draw, circle, inner sep=0pt, text width=2mm}]
    \node[vertex] (a) at (0,0) {};
    \node[vertex] (b) at (2,0) {};
    \draw[-{Classical TikZ Rightarrow[]}, bend left]  (a) to (b);
    \draw[-{Classical TikZ Rightarrow[]}, bend right] (a) to (b);
    \node (m) at (1,0.1) {$\vdots$};
    \node (m) at (1.2,0) {$m$};
  \end{tikzpicture}\,.
\end{equation}
We will write
\begin{equation}
  \kronecker{m}{(d,e)}\colonequals
  \modulispace[\theta_\can]{\mathrm{K}_m,(d,e)}
\end{equation}
for the associated \emph{Kronecker moduli},
a class of quiver moduli spaces with interesting properties
that has been studied in relation to moduli spaces of sheaves on~$\mathbb{P}^2$ \cite{MR0944602,MR0916199,MR0915184}
and more generally moduli spaces of sheaves on varieties \cite{MR2537070}
and modules over algebras \cite{MR1324213}.
Throughout we will assume that~$m\geq 3$ and~$\gcd(d,e)=1$.

\subsection{Invariants of Kronecker moduli}
\label{subsection:census}
We can identify many Kronecker moduli with one another
through so-called duality and periodicity for Kronecker moduli \cite[\S4.2, \S4.3]{MR0916199}.
\begin{proposition}[Duality and periodicity]
  \label{proposition:duality-periodicity}
  There exist natural isomorphisms~$\kronecker{m}{(d,e)}\cong\kronecker{m}{(e,d)}$
  and~$\kronecker{m}{(d,e)}\cong\kronecker{m}{(me-d,d)}$.
\end{proposition}
Thus, as explained in \cite[Proposition~6.2]{MR3683503}
it suffices to consider
\begin{equation}
  d\leq e\leq \frac{md}{2}.
\end{equation}
Moreover, if~$d=1$,
then we have an isomorphism~$\kronecker{m}{(1,e)}\cong\Gr(e,m)$,
whose geometric properties are well-known,
and we are not interested in this case.

\begin{proposition}
  \label{proposition:kronecker-invariants}
  The Kronecker moduli~$\kronecker{m}{(d,e)}$ is,
  if non-empty and not a point,
  a smooth projective Fano variety of dimension~$mde-d^2-e^2+1$,
  Picard rank~1
  and index~$m$.

  Its even Betti numbers,
  its degree of~$\mathcal{O}(1)$,
  and its Hilbert series of~$\mathcal{O}(1)$
  can be determined computationally.
  When~$mde-d^2-e^2+1<20$
  and~$\kronecker{m}{(d,e)}$ is not isomorphic to a Grassmannian
  there are precisely~4 such Kronecker moduli,
  and their invariants are given in \cref{table:kronecker-census}.
\end{proposition}

The overview in \cref{table:kronecker-census}
is of the same spirit as \cite[Appendice~1]{MR0916199}
which describes the smallest height-zero moduli spaces for~$\mathbb{P}^2$.

\begin{proof}
  The dimension, Picard rank and index
  are determined in \cite{MR0944602}
  and using more general methods involving quiver moduli in \cite[Corollary~5.2]{MR4352662}.
  For the Betti numbers one can use \cite[Corollary~6.9]{MR1974891}
  as implemented in~\cite{hodge-diamond-cutter}.
  The degree and Hilbert series can be computed
  using \cref{theorem:point-class,theorem:todd-class}.
  The Sage code to do this in a self-contained way is given in \cref{section:code}.

  Alternatively, Kronecker moduli are implemented in \cite{IntersectionTheory},
  as the \verb|matrix_moduli| function,
  and the Chow-theoretic invariants can thus be computed too.
  In op.~cit.~the description of
  the point class
  and the tangent bundle
  (and thus the Todd class)
  is obtained from \cite{MR1345086},
  see also \cref{remark:ellingsrud-stromme}.
\end{proof}

\begin{table}
  \centering
  \begin{tabular}{cccl}
    \toprule
    Kronecker moduli & dimension & degree & Betti numbers \\
    Hilbert series \\
    \midrule
    $\kronecker{3}{(2,3)}$ & 6 & 57 &
    $1, 1, 3, 3, 3, 1, 1$ \\
    \multicolumn{4}{l}{$1, 20, 148, 664, 2206, 5999, 14140, \ldots$} \\

    \addlinespace

    $\kronecker{4}{(2,3)}$ & 12 & 119020 &
    $1, 1, 3, 4, 7, 8, 10, 8, 7, 4, 3, 1, 1$ \\
    \multicolumn{4}{l}{$1, 126, 4032, 59268, 531839, 3395882, 16907632, \ldots$} \\

    \addlinespace

    $\kronecker{3}{(3,4)}$ & 12 & 1654983 &
    $1, 1, 3, 4, 7, 8, 10, 8, 7, 4, 3, 1, 1$ \\
    \multicolumn{4}{l}{$1, 266, 13222, 256438, 2779524, 20345430, 112317667, \ldots$} \\

    \addlinespace

    $\kronecker{5}{(2,3)}$ & 18 & 720578490 &
    $1, 1, 3, 4, 7, 9, 14, 16, 20, 20, 20, 16, 14, 9, 7, 4, 3, 1, 1$ \\
    \multicolumn{4}{l}{$1, 500, 51920, 2058485, 43370250, 585084682, 5666879250, \ldots$} \\
    \bottomrule
  \end{tabular}

  \label{table:kronecker-census}
  \caption{Invariants of the first 4 Kronecker moduli which are not Grassmannians}

  \begin{equation*}
    \begin{aligned}
      &\HS_{\mathcal{O}_{\kronecker{3}{(2,3)}}(1)}(t)(1-t)^7 \\
      &=t^4 + 13t^3 + 29t^2 + 13t + 1,
    \end{aligned}
  \end{equation*}
  \begin{equation*}
    \begin{aligned}
      &\HS_{\mathcal{O}_{\kronecker{4}{(2,3)}}(1)}(t)(1-t)^{13} \\
      &=t^9 + 113t^8 + 2472t^7 + 16394t^6 + 40530t^5 + 40530t^4 + 16394t^3 + 2472t^2 + 113t + 1,
    \end{aligned}
  \end{equation*}
  \begin{equation*}
    \begin{aligned}
      &\HS_{\mathcal{O}_{\kronecker{3}{(3,4)}}(1)}(t)(1-t)^{13} \\
      &=t^{10} + 253t^9 + 9842t^8 + 105014t^7 + 401785t^6 + 621193t^5 + 401785t^4 + 105014t^3 + 9842t^2 + 253t + 1,
    \end{aligned}
  \end{equation*}
  \begin{equation*}
    \begin{aligned}
      &\HS_{\mathcal{O}_{\kronecker{5}{(2,3)}}(1)}(t)(1-t)^{19} \\
      &=t^{14} + 481t^{13} + 42591t^{12} + 1156536t^{11} + 12656731t^{10} + 64666759t^9 + 167366129t^8 + 228800034t^7 \\
      &\qquad+ 167366129t^6 + 64666759t^5 + 12656731t^4 + 1156536t^3 + 42591t^2 + 481t + 1.
    \end{aligned}
  \end{equation*}
\end{table}

\begin{remark}
  \label{remark:vector-fields-and-rigidity}
  Using work-in-progress \cite{vector-fields,rigidity} it is moreover possible to show that
  \begin{equation}
    \begin{aligned}
      \hh^0(\kronecker{m}{(d,e)},\tangent_{\kronecker{m}{(d,e)}})&=m^2-1 \\
      \hh^{\geq 1}(\kronecker{m}{(d,e)},\tangent_{\kronecker{m}{(d,e)}})&=0
    \end{aligned}
  \end{equation}
  for~$(d,e)$ coprime,
  thus describing more invariants for Kronecker moduli.
  In particular,
  we know the size of the automorphism group,
  and that Kronecker moduli are rigid.
\end{remark}

\subsection{Identifying the Kronecker moduli space for \texorpdfstring{$m=3$}{m=3} and \texorpdfstring{$\mathbf{d}=(2,3)$}{d=(2,3)}}
The first Kronecker moduli space which is not a Grassmannian
is~$\kronecker{3}{(2,3)}$,
a~6\dash dimensional Fano variety of Picard rank~1 and index~3.
An important ingredient in the ongoing (re)classification of Fano varieties \cite{MR4482268},
and an important tool to determine their invariants,
is describing them using
zero loci of sections of an equivariant vector bundle
on a Grassmannian (or more generally a partial flag variety).
It is an interesting question to understand to which extent this can be done
for quiver moduli.
We will consider this for~$\kronecker{3}{(2,3)}$.

In \cite{MR1326986} K\"uchle classified Fano 4-folds of index~1
realised as zero loci of homogeneous vector bundles on Grassmannians.
The variety~(b11) in op.~cit.~is a codimension-2 linear section of the variety we are actually interested in.
\begin{construction}
  \label{construction:zero-locus}
  Denote by~$X$ the zero locus of a general section of~$\mathcal{Q}^\vee(1)$ on~$\Gr(2,8)$.
\end{construction}
Here~$\mathcal{Q}$ is the universal quotient bundle of rank~6.
Using this description,
which allows Borel--Weil--Bott and Chow ring calculations,
the invariants listed in \cref{table:kronecker-census}
can computed as in the proof of \cite[Theorem~4.8]{MR1326986},
and they are seen to agree with those of~$\kronecker{3}{(2,3)}$.
And indeed, they are isomorphic.
\begin{proposition}
  \label{proposition:identification}
  The variety~$X$ from \cref{construction:zero-locus}
  is isomorphic to~$Y\colonequals\kronecker{3}{(2,3)}$.
\end{proposition}

\begin{proof}
  By \cite[Th\'eor\`eme~4]{MR0944602}
  we have that~$\operatorname{Hilb}^3\mathbb{P}^2$
  is the blowup of~$Y$ in a~$\mathbb{P}^2$,
  i.e.,
  it is the second elementary contraction
  (the first one being the Hilbert--Chow morphism to~$\operatorname{Sym}^3\mathbb{P}^2$).
  By \cite[Theorem~4.2]{MR2164624}
  we can thus identify~$X$ and~$Y$,
  because the variety~$Y_2$ in op.~cit.
  is the zero section~$X$ by using \cite[Theorem~3.8]{MR3732686}.
\end{proof}
This raises the question:
are there similar identifications for other Kronecker moduli?
The next cases to consider can be found in \cref{table:kronecker-census}.

The Kronecker moduli~$X$
is isomorphic to the height-zero moduli space~$Z=\mathrm{M}_{\mathbb{P}^2}(4,-1,3)$
of stable sheaves on~$\mathbb{P}^2$ with~$(r,\mathrm{c}_1,\mathrm{c}_2)=(4,-1,3)$
by \cite[Th\'eor\`eme~2]{MR0916199}.
This interpretation was essential for the proof of \cref{proposition:identification}.
The other Kronecker moduli in \cref{table:kronecker-census}
are not of this form:
for this to hold we need that~$m$ is a multiple of~3,
by \cite[Th\'eor\`eme~2]{MR0916199}.

\begin{remark}
  \label{remark:ellingsrud-stromme}
  The next case,
  that of~$\kronecker{4}{(2,3)}$
  was studied in \cite{MR1345086}
  in order to determine the number of twisted cubics on a general quintic 3-fold.
  This was an important litmus test for enumerative mirror symmetry of Calabi--Yau 3-folds.
  The intersection theory results in op.~cit.~are special cases of the general results in our paper,
  e.g.,
  the 4-term tangent bundle sequence in \eqref{equation:4-term-tangent-sequence}
  corresponds to Equation (4--4) in op.~cit.
\end{remark}

\newpage
\appendix

\section{Code for \texorpdfstring{\cref{proposition:kronecker-invariants}}{Proposition \ref{proposition:kronecker-invariants}}}
\label{section:code}
\inputminted[fontsize=\small]{sage}{kronecker.sage}

\printbibliography

\texttt{pieter.belmans@uni.lu} \\
University of Luxembourg, Department of Mathematics,
6, Avenue de la Fonte,
L-4364 Esch-sur-Alzette,
Luxembourg

\texttt{hans.franzen@math.upb.de} \\
Paderborn University, Institute of Mathematics,
Warburger Stra\ss e 100,
33098 Paderborn,
Germany

\end{document}